\documentclass[a4paper,12pt]{amsart}
\usepackage{amsfonts}
\usepackage{amsmath,amssymb,amsthm}

\everymath{\displaystyle}

\newcommand{\E}{\mathbb E}
\newcommand{\R}{\mathbb R}
\newcommand{\tr}{\mathrm{tr}}
\newcommand{\ds}{\displaystyle}

\newtheorem{thm}{Theorem}[section]

\newtheorem{prop}[thm]{Proposition}
\theoremstyle{definition}

\theoremstyle{remark}

\begin{document}

\title[Minimal Lorentz Surfaces in Pseudo-Euclidean 4-Space]
{Minimal Lorentz Surfaces in Pseudo-Euclidean 4-Space with Neutral Metric}

\author{Yana  Aleksieva, Velichka Milousheva}

\address{Faculty of Mathematics and Informatics, Sofia University,
5 James Bourchier blvd., 1164 Sofia, Bulgaria}
\email{yana\_a\_n@fmi.uni-sofia.bg}

\address{Institute of Mathematics and Informatics, Bulgarian Academy of Sciences,
Acad. G. Bonchev Str. bl. 8, 1113, Sofia, Bulgaria}
\email{vmil@math.bas.bg}


\begin{abstract}
We study minimal Lorentz surfaces in the pseudo-Euclidean 4-space with neutral metric whose first normal space is two-dimensional and whose Gauss curvature $K$ and normal curvature $\varkappa$ satisfy the inequality $K^2-\varkappa^2 >0$. Such surfaces we call minimal Lorentz surfaces of general type. On any surface of this class 
we introduce geometrically determined canonical parameters and prove that any minimal Lorentz surface of general type is determined (up to a rigid motion) by two invariant functions satisfying  a system of two natural
partial differential equations. Using a concrete  solution to this system we construct an example of a minimal Lorentz surface of general type.
\end{abstract}

\maketitle

\section{Introduction}

The study of minimal surfaces is one of the main topics in classical differential geometry which  goes back to the latter part of the 18th century. Lagrange was the first who  initiated in 1760 the study of minimal surfaces in Euclidean 3-space and found the minimal surface equation  when he looked  for a necessary condition for minimizing a certain integral. Actually, the notion of mean curvature was first formally defined by Meusnier in 1776. Throughout the 19th century grate mathematicians such as Gauss and Weierstrass devoted much of their studies to these surfaces. The theory of minimal surfaces in real space forms have been attracting the attention of many mathematicians for more than two  centuries (see \cite{Chen-book}, \cite{Nitsche}, and the references therein).  

In  the last years, great attention is also paid to Lorentz surfaces in pseudo-Euclidean spaces, since pseudo-Riemannian geometry has many important applications in Physics. 
Minimal Lorentz surfaces in $\mathbb{C}^2_1$ have been classified recently by  B.-Y. Chen \cite{Chen-TaJM}.
Classification results for minimal Lorentz surfaces in  pseudo-Euclidean space $\E^m_s$ with arbitrary dimension $m$ and arbitrary index $s$  are obtained in \cite{Chen}.

Metrics of neutral signature $(+,+,-,-)$  in dimension four appear in many geometric and physics problems  as well as  in string theory.
In the present paper we study minimal Lorentz surfaces in the pseudo-Euclidean 4-space with neutral metric $\E^4_2$.  
Our aim is to characterize the minimal Lorentz surfaces in terms of a pair of smooth functions satisfying a system of two natural
partial differential equations.  This aim is motivated by similar results concerning  minimal surfaces in the four-dimensional Euclidean space $\E^4$  (see \cite{TrGua}) and spacelike or timelike surfaces with zero mean curvature in the Minkowski 4-space $\E^4_1$ (see  \cite{AP} and \cite{GM}). Our approach to the study of minimal Lorentz surfaces in $\E^4_2$ is based on the introducing of special geometric parameters which we call canonical parameters.

It is  known that the Gauss curvature $K$, the curvature of the normal connection $\varkappa$, and the mean curvature vector field $H$ of an arbitrary surface in the Euclidean space $\E^4$ satisfy the Wintgen inequality  \cite{Wintgen}
$$K + |\varkappa| \leq ||H||^2.$$  Following \cite{P-T&V}, a surface in $\E^4$ is called
\emph{Wintgen ideal surface}, if it satisfies the equality case of
the Wintgen's inequality identically. A surface $M$ is called \textit{minimal}, if its mean curvature vector  vanishes identically. 
Therefore, the invariants $K$ and $\varkappa$  of any minimal surface in $\E^4$ satisfy the inequality
$K^2-\varkappa^2\geq 0$, which divides the minimal surfaces into two basic classes:
\begin{itemize}
\item the class of minimal Wintgen ideal  surfaces characterized by
$K^2 - \varkappa^2 =0$; \item the class of minimal surfaces of general type characterized
by $K^2-\varkappa^2>0$.
\end{itemize}
The Wintgen ideal surfaces
are characterized by circular ellipse of normal curvature
\cite{Guad-Rod}. A surface $M$ in $\E^4$ is called \textit{super-conformal} \cite{BFLPP} if at any point of $M$
the ellipse of curvature is a circle. Hence, the class of minimal Wintgen ideal surfaces coincides with the
class of minimal super-conformal surfaces.
According to a result of Eisenhart \cite{Ein},  the class of minimal super-conformal surfaces in $\E^4$ is locally equivalent to the class
of holomorphic curves in  $\mathbb{C}^2 \equiv \E^4$.

In \cite{TrGua}, de Azevedo Tribuzy and Guadalupe proved that the Gauss curvature $K$ and the curvature of the normal connection
$\varkappa$ of any minimal non-super-conformal surface parametrized by special isothermal parameters in the Euclidean space $\E^4$ satisfy
the following system of  partial differential equations
\begin{equation*}
\begin{array}{l}
\ds{(K^2 - \varkappa^2)^{\frac{1}{4}}\, \Delta \ln (K^2 - \varkappa^2)} = 8 K\\
[2mm]
\ds{(K^2 - \varkappa^2)^{\frac{1}{4}}\, \Delta \ln \frac{K-\varkappa}{K + \varkappa}= -4 \varkappa}
\end{array}
\end{equation*} 
and conversely, any solution ($K$, $\varkappa$) to this system determines  a unique (up to a rigid
motion in $\E^4$) minimal non-super-conformal surface with Gauss
curvature $K$ and normal curvature $\varkappa$.

Similar results hold for surfaces with zero mean curvature in the 4-dimensional Minkowski space $\E^4_1$.
In \cite{AP},  Al\'{i}as and Palmer studied  spacelike surfaces with zero mean curvature  in $\E^4_1$ and proved that these surfaces are described by the following  system of  partial differential equations
\begin{equation*}
\begin{array}{l}
\ds{(K^2 + \varkappa^2)^{\frac{1}{4}}\, \Delta \ln (K^2 + \varkappa^2)} = 8 K \\
[2mm] \ds{(K^2 + \varkappa^2)^{\frac{1}{4}}\, \Delta \arctan
\frac{\varkappa}{K} = 2 \varkappa}
\end{array}
\end{equation*}
where $K$ and $\varkappa$ are the Gauss curvature and the normal
curvature, respectively.

In \cite{GM}, Ganchev and the second author developed the local theory of timelike surfaces with zero mean curvature in the 
Minkowski 4-space $\E^4_1$ and proved that the Gauss curvature $K$ and
the normal curvature $\varkappa$ of any  timelike  surface (parametrized by special, so called  canonical parameters) with zero mean curvature 
satisfy the following system of natural partial differential equations
\begin{equation}\notag
\begin{array}{l}
\ds{(K^2 + \varkappa^2)^{\frac{1}{4}}\, \Delta^h \ln (K^2 + \varkappa^2)} = 8 K \\
[2mm] \ds{(K^2 + \varkappa^2)^{\frac{1}{4}}\, \Delta^h \arctan
\frac{\varkappa}{K} = 2 \varkappa}
\end{array}
\end{equation}
where $\Delta^h$ denotes the hyperbolic Laplace operator.  Conversely, any solution ($K$, $\varkappa$) to the above
system, determines a unique (up to a rigid  motion in $\E^4_1$)  timelike
surface  with zero mean curvature such that $K$ is the  Gauss
curvature and $\varkappa$ is the normal curvature of the surface.

 Minimal Lorentz surfaces in the pseudo-Euclidean space $\E^4_2$ can  be divided into three 
 basic classes:
\begin{itemize}
\item surfaces characterized by $K^2 - \varkappa^2 > 0$; 
\item  surfaces characterized by  $K^2-\varkappa^2 = 0$;
\item surfaces characterized by  $K^2-\varkappa^2 < 0$.
\end{itemize}

In the present paper we study  minimal Lorentz  surfaces  from the first class, i.e. satisfying the inequality  $K^2 - \varkappa^2 > 0$. In the special case when the first normal space is one-dimensional at each point, the minimal surface is either a degenerate hyperplane, a flat umbilic surface, a quasi-umbilic surface, or lies in a non-degenerate hyperplane. We focus our attention on minimal Lorentz surfaces whose first normal space is two-dimensional and whose Gauss curvature $K$ and normal curvature $\varkappa$ satisfy $K^2 - \varkappa^2 > 0$ on a dense open subset. We call these surfaces \textit{minimal  surfaces of general type}, since their  local theory  can be developed analogously to the local theory of minimal surfaces in $\E^4$ and  spacelike or timelike surfaces with zero mean curvature vector  in $\E^4_1$.
 We introduce special geometric (canonical) parameters $(u,v)$ on any minimal Lorentz surface of general type.
With respect to these parameters all coefficients of the first and the second fundamental form are expressed by two invariant functions $\mu(u,v)$ and $\nu(u,v)$. Using the canonical parameters we introduce a geometrically determined moving frame field $\{x, y, n_1, n_2\}$ at each point of the surface, where $x$ and $y$ are unit tangent vector fields determining the so-called   canonical directions of the surface; $n_1$ and $n_2$ are  unit normal vector fields which are uniquely determined by the canonical tangents. We prove a fundamental existence and uniqueness theorem (Theorem \ref{T:Fundamental}) stating that any minimal Lorentz surface of general type is determined up to a rigid motion in $\E^4_2$ by the invariants  $\mu(u,v)$ and $\nu(u,v)$ satisfying the following system of natural partial differential equations 
\begin{equation} \label{E:Int-1}
\begin{array}{l}
\vspace{2mm} 
\ds{\sqrt{\left|\mu^2 - \nu^2\right|}}\,\Delta^h \ln \left|\mu^2 - \nu^2\right| = - 4\varepsilon (\mu^2 + \nu^2)\\
\vspace{2mm} 
\ds{\sqrt{\left|\mu^2 - \nu^2\right|}}\,\Delta^h \ln\left|\frac{\mu + \nu}{\mu - \nu}\right| = -4\,\varepsilon\mu\,\nu
\end{array} \qquad \quad (\varepsilon = \pm 1),
\end{equation}
where $\Delta^h = \frac{\partial^2}{\partial u^2} - \frac{\partial^2}{\partial v^2}$ is the hyperbolic Laplace operator, $\varepsilon = 1$ corresponds to the case the geometric normal vector field $n_1$ is spacelike, $\varepsilon = -1$ corresponds to the case  $n_1$ is timelike.

Further, expressing the Gauss curvature $K$ and the normal curvature $\varkappa$ by the invariants $\mu$ and $\nu$, we prove that  $K$ and $\varkappa$  satisfy the following  system of natural
partial differential equations
\begin{equation} \label{E:Int-2}
\begin{array}{l}
\vspace{2mm} 
\ds{\sqrt[4]{K^2 - \varkappa^2}}\,\Delta^h \ln \left(K^2 - \varkappa^2 \right)  = 8 K\\
\vspace{2mm} 
\ds{\sqrt[4]{K^2 - \varkappa^2}}\,\Delta^h \ln \frac{K + \varepsilon\,\varkappa}{K - \varepsilon\,\varkappa}  = 4 \varepsilon \varkappa
\end{array} 
\end{equation}
Conversely, any solution $(K, \varkappa)$  to this system
determines a unique (up to a rigid  motion in $\E^4_2$) minimal Lorentz surface of general type with Gauss curvature $K$ and normal curvature $\varkappa$ and such that the given parameters are canonical. 

The above system is the background system of partial differential equations describing the class of minimal Lorentz surfaces of general type. Equalities  \eqref{E:Int-2} follow also from a result of M. Sakaki \cite{Sa} for Lorentz stationary surfaces in 4-dimensional space forms of index 2. We obtain system  \eqref{E:Int-2}  for the invariants $K$ and $\varkappa$ as a consequence of  system \eqref{E:Int-1} for the invariants $\mu$ and $\nu$. We prefer to describe the class of minimal Lorentz surfaces of general type by the system for the invariants $\mu$ and $\nu$ since the proof of Theorem \ref{T:Fundamental} gives a procedure for constructing  a minimal surface given  a solution $(\mu, \nu)$ to  system \eqref{E:Int-1}. As an application of this procedure, in the last section we obtain an example of a minimal Lorentz surface of general type using a concrete solution $(\mu,\nu)$ to system \eqref{E:Int-1}.

The class of minimal Lorentz surfaces satisfying the equality $K^2 - \varkappa^2 = 0$ is  the analogue of the class of minimal super-conformal surfaces in the Euclidean space $\E^4$. Minimal surfaces from the third class, i.e. satisfying the inequality  $K^2 - \varkappa^2 < 0$ need a different approach to be studied with. Our idea for introducing  a geometric frame field based on the canonical directions cannot be applied for this class of minimal surfaces, since in the case $K^2 - \varkappa^2 < 0$ there do not exist canonical (in our sense) directions. Note that such class of minimal surfaces does not exist neither in the Euclidean space $\E^4$ nor in the Minkowski space $\E^4_1$.

\section{Preliminaries}

Let  $\E^4_2$ be the pseudo-Euclidean 4-space endowed  with the canonical pseudo-Euclidean metric of index 2 given in local coordinates by
$$g_0 = dx_1^2 + dx_2^2 - dx_3^2 - dx_4^2,$$
where $(x_1, x_2, x_3, x_4)$ is a rectangular coordinate system of $\E^4_2$. 
We denote by $\langle .\, , . \rangle$ the indefinite inner scalar product associated with $g_0$. Since $g_0$ is an indefinite metric, a vector 
 $v \in \E^4_2$  can have one of the three casual characters:  \textit{spacelike}, if  $\langle v, v \rangle > 0$ or $v = 0$,   
\textit{timelike} if $\langle v, v \rangle < 0$, and
\textit{lightlike} if  $\langle v, v \rangle = 0$ and $v \neq 0$.
This terminology is inspired by  General Relativity and the Minkowski 4-space $\E^4_1$.

A surface $M^2_1$ in $\E^4_2$ is called  \textit{Lorentz} if the induced  metric $g$ on $M^2_1$ is Lorentzian, i.e. 
at each point $p \in M^2_1$ we have the following decomposition
$$\E^4_2 = T_pM^2_1 \oplus N_pM^2_1$$
with the property that the restriction of the metric
onto the tangent space $T_pM^2_1$ is of
signature $(1,1)$, and the restriction of the metric onto the normal space $N_pM^2_1$ is of signature $(1,1)$.

Denote by $\nabla$ and $\nabla'$ the Levi Civita connections of $M^2_1$  and $\E^4_2$, respectively.
Let $x$ and $y$ be vector fields tangent to $M^2_1$ and $\xi$ be a normal vector field.
The formulas of Gauss and Weingarten are given respectively by
$$\begin{array}{l}
\vspace{2mm}
\nabla'_xy = \nabla_xy + \sigma(x,y);\\
\vspace{2mm}
\nabla'_x \xi = - A_{\xi} x + D_x \xi,
\end{array}$$
where $\sigma$ is the second fundamental form of $M^2_1$, $D$ is the normal
connection on the normal bundle, and $A_{\xi}$ is the shape operator  with respect to
$\xi$. In general, $A_{\xi}$ is not diagonalizable.
The mean curvature vector  field $H$ of   $M^2_1$ in $\E^4_2$ 
is defined as 
$$H = \frac{1}{2}\,  \tr\, \sigma.$$ 
A Lorentz surface in an indefinite space form is called \textit{totally geodesic} if its
second fundamental form vanishes identically. It is called \textit{minimal} if its
mean curvature vector vanishes identically, i.e.  $H=0$.

\section{Minimal Lorentz surfaces whose first normal space is one-dimensional}

Let $M^2_1$ be a  Lorentz surface in $\E^4_2$. According to a result of 
 Larsen \cite{Larsen}, locally there exist isothermal coordinates $(u,v)$ such that  the metric tensor $g$ of $M^2_1$ takes the form
\begin{equation} \label{E:Eq-g}
g= f^2(u, v)(du\otimes du - dv\otimes dv)
\end{equation}
for some positive function $f(u, v)$. Let $z=z(u, v), (u, v) \in \mathcal{D} \, 
 (\mathcal{D} \subset \R^2)$ be a local parametrization on $M^2_1$ with respect to such isothermal parameters.
Then, the coefficients of the first fundamental form are
\begin{equation*}
E = \langle z_u, z_u \rangle = f^2(u, v), \quad F = \langle z_u, z_v \rangle = 0, \quad G = \langle z_v, z_v \rangle = -f^2(u, v).
\end{equation*}
We consider the orthonormal tangent frame field  given by $x=\ds{\frac{z_u}{f}}$, $y=\ds{\frac{z_v}{f}}$. Obviously,
 $\langle x, x \rangle=1$,  $\langle x, y \rangle=0$, $\langle y, y \rangle=-1$. It can easily be checked that the Levi-Civita connection of the metric tensor \eqref{E:Eq-g} satisfies 
\begin{equation} \label{E:Eq-1-a}
\begin{array}{l}
\vspace{2mm}
\nabla_{z_u} z_u= \frac{f_u}{f} \,z_u + \frac{f_v}{f} \, z_v = f_u \,x + f_v \,y;\\
\vspace{2mm}
\nabla_{z_u} z_v= \frac{f_v}{f} \, z_u + \frac{f_u}{f} \, z_v = f_v \,x + f_u \,y;\\
\vspace{2mm}
\nabla_{z_v} z_v= \frac{f_u}{f} \,z_u + \frac{f_v}{f} \,z_v = f_u \,x + f_v \,y.\\
\end{array}
\end{equation}
Then equalities \eqref{E:Eq-1-a} imply the following derivative formulas 
\begin{equation} \label{E:Eq-2-a}
\begin{array}{l}
\vspace{2mm}
\nabla'_xx= \quad\quad\; \frac{f_v}{f^2}\,y + \sigma (x, x);\\
\vspace{2mm}
\nabla'_xy= \frac{f_v}{f^2}\,x \qquad \; + \sigma (x, y);\\
\vspace{2mm}
\nabla'_yx= \quad\quad\; \frac{f_u}{f^2}\,y + \sigma (x, y);\\
\vspace{2mm}
\nabla'_yy= \frac{f_u}{f^2}\,x \qquad \; + \sigma (y, y).
\end{array}
\end{equation}

The mean curvature vector field  of $M^2_1$ is given by $H = \ds{\frac{\sigma(x, x) - \sigma(y, y)}{2}}$, since  $\langle x, x \rangle=1$, $\langle y, y \rangle=-1$. Hence, $M^2_1$ is a minimal surface in $\E^4_2$ if and only if $\sigma(x, x) = \sigma(y, y)$.

\vskip 2mm
The classification of minimal surfaces in an arbitrary indefinite pseudo-Euclidean space $\E^m_s$  is given by B.-Y. Chen in the next theorem.

\begin{thm}\label{T:minimal} \cite{Chen}
A Lorentz surface in a pseudo-Euclidean m-space $\E^m_s$ is minimal if and only if, locally the surface is parametrized by
$$z(u,v) = \alpha(u) + \beta(v),$$
where $\alpha$ and $\beta$ are null curves satisfying $\left\langle \alpha'(u), \beta'(v) \right\rangle \neq 0$.
\end{thm}

\vskip 2mm
We study minimal Lorentz surfaces in $\E^4_2$ in terms of geometric (canonical) parameters and geometric frame field which allow us to determine each minimal  surface of  general type by two  smooth functions satisfying a system of two natural partial differential equations.

\vskip 2mm
Let $M^2_1$ be a minimal surface parametrized by isothermal parameters such that the metric tensor $g$ is given by \eqref{E:Eq-g} and hence  formulas \eqref{E:Eq-2-a} hold true. Since $M^2_1$ is minimal, we have $\sigma(x, x) = \sigma(y, y)$. 

At a given point $p \in M^2_1$, the \textit{first normal space} of $M^2_1$  in $\E^4_2$, denoted by
$\rm{Im} \, \sigma_p$, is the subspace given by
$${\rm Im} \, \sigma_p = {\rm span} \{\sigma(X, Y): X, Y \in T_p M^2_1 \}.$$

In this section we consider minimal Lorentz surfaces for which the first normal space  at each point is one-dimensional. 
A point $p$ of such surface is called \textit{degenerate} if the Gauss curvature $K$ and the curvature of
the normal connection $\varkappa$  (the normal curvature) are both zero at $p$. According to a result in \cite{BaPaBringas}, a minimal Lorentz surface consisting of degenerate points either belongs to a degenerate hyperplane or is a flat umbilic or quasi-umbilic surface. 

In the next theorem we describe the minimal Lorentz surfaces  whose first normal space is one-dimensional.

\begin{thm}
Let $M^2_1$ be a minimal Lorentz surface in $\E^4_2$ for which the first normal space  at each point is one-dimensional. Then in a neighbourhood of a non-degenerate point  $M^2_1$ is  a non-flat surface lying in a non-degenerate hyperplane of $\E^4_2$.
\end{thm}

\begin{proof}
Since the first normal space is one-dimensional, at least locally there exist  a normal vector field $n$ and smooth functions $\nu$, $\mu$ such that 
\begin{equation*}
\begin{array} {l}
\vspace{2mm}
\sigma(x, x) = \sigma(y, y) = \nu \,n;\\
\vspace{2mm}
\sigma(x, y) = \mu \,n.
\end{array}
\end{equation*}

Note that if $\nu = \mu =0$ then the surface $M^2_1$ is totally geodesic, and hence $M^2_1$ is  contained in a two-dimensional plane $\E^2_1$. 
So, we assume that at least one of the functions  $\nu$, $\mu$ is non-zero.

Let $\{e_1, e_2\}$ be a local normal frame field of $M^2_1$, defined in $\mathcal{D}$, such that $\langle e_1, e_1 \rangle = 1$, $\langle e_2, e_2 \rangle = -1$, $\langle e_1, e_2 \rangle = 0$ and $n = a\, e_1 + b\, e_2$ for some smooth functions $a$ and $b$, where $a^2 +b^2 \neq 0$. Then 
 with  respect to the  frame field $\{x, y, e_1, e_2\}$ we have the following Frenet-type derivative formulas
\begin{equation} \label{E: Frenet1}
\begin{array} {ll}
\vspace{2mm}
\nabla'_{x}x=\quad\quad\;-\gamma_{1}\,y+\mu a\,e_1 + \mu b\,e_2; & \qquad\quad\nabla'_{x}e_1=-\mu a\,x+\nu a\,y\quad\,\quad\quad+\beta_{1}\,e_2;\\
\vspace{2mm}
\nabla'_{x}y=-\gamma_{1}\,x\quad\quad\;+\nu a\,e_1 + \nu b\,e_2; & \qquad\quad\nabla'_{y}e_1=-\nu a\,x+\mu a\,y\quad\,\quad\quad+\beta_{2}\,e_2;\\
\vspace{2mm}
\nabla'_{y}x=\quad\quad\;-\gamma_{2}\,y+\nu a\,e_1 + \nu b\,e_2; & \qquad\quad\nabla'_{x}e_2=\mu b\,x -\nu b\,y + \beta_{1}\,e_1;\\
\vspace{2mm}
\nabla'_{y}y=-\gamma_{2}\,x\quad\quad\;+\mu a\,e_1 + \mu b\,e_2; & \qquad\quad\nabla'_{y}e_2=\nu b\,x -\mu b\,y + \beta_{2}\,e_1,
\end{array}
\end{equation}
where $\gamma_1 = - \frac{f_v}{f^2}$, $\gamma_2 = - \frac{f_u}{f^2}$, $\beta_1$ and $\beta_2$ are  smooth functions determining the components of the normal connection.

Since the Levi-Civita connection $\nabla'$ is flat, from $R'(x,y,x) = 0, R'(x,y,y) = 0, R'(x,y,e_1) = 0$, using \eqref{E: Frenet1} we obtain that the functions 
$\gamma_1$,  $\gamma_2$, $\nu$, $\mu$, $\beta_1$, $\beta_2$, $a$, $b$  satisfy  the following conditions 
\begin{equation} \label{E:Int-Cond}
\begin{array} {l}
\vspace{2mm} x(\nu a) - y(\mu a) = 2\nu a\gamma_2 - 2\mu a\gamma_1 -  \nu b\beta_1 + \mu b\beta_2;\\
\vspace{2mm} x(\nu b) - y(\mu b) = 2\nu b\gamma_2 - 2\mu b\gamma_1 -  \nu a\beta_1 + \mu a\beta_2;\\
\vspace{2mm} x(\mu a) - y(\nu a) = 2\mu a\gamma_2 - 2\nu a\gamma_1 -  \mu b\beta_1 + \nu b\beta_2;\\
\vspace{2mm} x(\mu b) - y(\nu b) = 2\mu b\gamma_2 - 2\nu b\gamma_1 -  \mu a\beta_1 + \nu a\beta_2;\\
\vspace{2mm} x(\beta_2) - y(\beta_1) + \gamma_1\beta_1 - \gamma_2\beta_2 = 0.
\end{array}
\end{equation}

The Gauss curvature $K$ and the normal curvature $\varkappa$ are given by the following formulas
\begin{equation*}
\begin{array} {l}
\vspace{6mm} K = \ds{\frac{\langle \sigma(x, x), \sigma(y, y) \rangle - \langle \sigma(x, y), \sigma(x, y) \rangle}{\langle x, x \rangle\,\langle y, y \rangle -  \langle x, y \rangle^2}};\\
\vspace{2mm} \varkappa = \ds{\frac{\langle R^{\bot}(x, y)\,n_1, n_2 \rangle}{\langle x, x \rangle\,\langle y, y \rangle -  \langle x, y \rangle^2}},
\end{array}
\end{equation*}
where 
$R^{\bot}(x, y)\,n_1  = D_x D_y\, n_1 - D_y D_x\, n_1 - D_{[x, y]}\,n_1$.
So, using \eqref{E: Frenet1} we obtain the following expressions for the Gauss curvature and the normal curvature of the surface
\begin{equation*} 
 K = - (\mu^2 - \nu^2) (a^2 - b^2);
\end{equation*}
\begin{equation} \label{E:Eq-kappa}
 \varkappa = x(\beta_2) - y(\beta_1) + \gamma_1\beta_1 - \gamma_2\beta_2.
\end{equation}

The expression for the Gauss curvature can be established also from the choice of the parameters and $e_1, e_2$.
Note that  \eqref{E:Eq-kappa} and the last equality of \eqref{E:Int-Cond}  imply  $\varkappa = 0$, i.e. the surface has flat normal connection. Hence, points of the surface at which $\mu^2 - \nu^2 = 0$ or $a^2 - b^2 = 0$  are degenerate.

In a neighbourhood $\mathcal{D}_0 \subset \mathcal{D}$ of a non-degenerate point we have $\mu^2 - \nu^2 \neq 0$ and $a^2 - b^2 \neq 0$. From the first four equalities of \eqref{E:Int-Cond} we obtain 
\begin{equation*} 
\begin{array} {l}
\vspace{2mm}
(\mu^2 - \nu^2) (b x(a) - a x(b)) = \beta_1 (\mu^2 - \nu^2)(a^2 - b^2);\\
\vspace{2mm}
(\mu^2 - \nu^2) (b y(a) - a y(b)) = \beta_2 (\mu^2 - \nu^2)(a^2 - b^2).
\end{array}
\end{equation*}
The last two equalities imply
\begin{equation} \label{E:Eq-beta}
\beta_1 = \frac{b x(a) - a x(b)}{a^2 - b^2};\qquad 
\beta_2 = \frac{b y(a) - a y(b)}{a^2 - b^2}.
\end{equation}
Let us consider the normal vector field $\bar{n}$ defined by
\begin{equation*} 
\bar{n} = \frac{b}{\sqrt{|a^2 - b^2|}}\; e_1 + \frac{a}{\sqrt{|a^2 - b^2|}}\; e_2.
\end{equation*}
Since $a^2 - b^2 \neq 0$, the vector fields $\bar{n}$  and $n$ are   non-lightlike in $\mathcal{D}_0$. Obviously, $\bar{n}$ is orthogonal to $n$.
Using \eqref{E: Frenet1} and \eqref{E:Eq-beta} we calculate that $\nabla'_{x}\bar{n} = 0$, $\nabla'_{y}\bar{n} = 0$,  i.e. the normal vector field $\bar{n}$ is constant.
Hence,  the surface $M^2_1{/_{\mathcal{D}_0}}$ lies in a constant 3-dimensional space parallel to ${\rm span} \{x, y, n\}$. Moreover, since $K \neq 0$, the surface is non-flat. Consequently, $M^2_1{/_{\mathcal{D}_0}}$ is a non-flat surface lying  in a hyperplane $\E^3_1$ or $\E^3_2$ of $\E^4_2$. 

\end{proof}

\section{Minimal Lorentz surfaces of general type}

In this section we study minimal Lorentz surfaces for which the first normal space is two-dimensional at each point. 

Let $M^2_1$ be a minimal surface parametrized by isothermal parameters such that the metric tensor $g$ is given by \eqref{E:Eq-g}. 
We choose a local  normal frame field  $\{e_1, e_2\}$, defined in $\mathcal{D}$, such that  $\langle e_1, e_1 \rangle = 1$, $\langle e_1, e_2 \rangle = 0$,  $\langle e_2, e_2 \rangle = -1$. Then the components of the second fundamental form are expressed as follows
\begin {equation*} 
\begin {array} {l}
\vspace{2mm}
\sigma(x, x) =  a \, e_1 + b \, e_2;\\
\vspace{2mm}
\sigma(x, y) = c\, e_1 + d \,e_2;\\
\vspace{2mm}
\sigma(y, y) = a \, e_1 + b \, e_2
\end {array}
\end {equation*}
for some smooth functions  $a(u,v)$, $b(u,v)$, $c(u,v)$, and  $d(u,v)$. 

Let us note that $a\,d - b\,c \neq 0$, $a^2 + b^2 \neq 0$, $c^2 + d^2 \neq 0$, since the first normal space is two-dimensional.

With  respect to the  frame field $\{x, y, e_1, e_2\}$ we have the following  derivative formulas
\begin{equation} \label{E:Eq-1}
\begin{array} {ll}
\vspace{2mm}
\nabla'_{x}x=\quad\quad\;-\gamma_{1}\,y+ a\,e_1 + b\,e_2; & \qquad\quad\nabla'_{x}e_1=-a\,x+c\,y\quad \quad\quad+\beta_{1}\,e_2;\\
\vspace{2mm}
\nabla'_{x}y=-\gamma_{1}\,x\quad\quad\;+c\,e_1 + d\,e_2; & \qquad\quad\nabla'_{y}e_1=-c\,x+a\,y\quad \quad\quad+\beta_{2}\,e_2;\\
\vspace{2mm}
\nabla'_{y}x=\quad\quad\;-\gamma_{2}\,y+c\,e_1 + d\,e_2; & \qquad\quad\nabla'_{x}e_2=b\,x-d\,y\quad  +\beta_{1}\,e_1;\\
\vspace{2mm}
\nabla'_{y}y=-\gamma_{2}\,x\quad\quad\;+a\,e_1 + b\,e_2; & \qquad\quad\nabla'_{y}e_2=d\,x-b\,y \quad +\beta_{2}\,e_1.
\end{array}
\end{equation}

It follows from equalities \eqref{E:Eq-1} that the Gauss curvature $K$ is expressed as
$$K = b^2 - a^2 + c^2 - d^2.$$
Using the equation of Ricci and formulas  \eqref{E:Eq-1}  we get that the normal curvature is 
$$\varkappa = 2(bc-ad).$$
Consequently, 
\begin{equation} \label{E:Eq-2}
K^2 - \varkappa^2 = (a^2 - b^2)^2 + (c^2 - d^2)^2 - 2(bc-ad)^2 - 2(ac-bd)^2.
\end{equation}

We shall consider minimal surfaces whose first normal space is two-dimensional and whose Gauss curvature $K$ and normal curvature $\varkappa$ satisfy the following inequality
$$K^2 - \varkappa^2  > 0$$
on a dense open subset of $M^2_1$.
Such minimal surfaces we call \textit{minimal surfaces of general type}. For this class of surfaces 
 we shall introduce a  local orthonormal  frame field 
$\{x, y, n_1, n_2\}$  such that the vector fields $\sigma(x, x)$ and $\sigma(x, y)$ are collinear to $n_1$ and $n_2$, respectively, i.e.
\begin {equation*}
\begin {array} {l}
\vspace{2mm}\sigma(x, x) = \nu\,n_1;\\
\vspace{2mm}\sigma(x, y) = \mu\,n_2
\end {array}
\end {equation*}
for some smooth functions $\nu(u,v) \neq 0$ and $\mu(u,v) \neq 0$.

\vskip 3mm
Using that $R'(x,y,x) = 0$, $R'(x,y,y) = 0$, $R'(x,y,e_1) = 0$, from \eqref{E:Eq-1} we obtain that the functions 
$\gamma_1$,  $\gamma_2$, $\beta_1$,  $\beta_2$, $a$, $b$, $c$, $d$ satisfy the following conditions 
\begin{equation} \label{E:Int-Cond-1}
\begin{array} {l}
\vspace{2mm} x(c) - y(a) = 2c\gamma_2 - 2a\gamma_1 -  d\beta_1 + b \beta_2;\\
\vspace{2mm} x(d) - y(b) = 2d\gamma_2 - 2b\gamma_1 -  c\beta_1 + a\beta_2;\\
\vspace{2mm} x(a) - y(c) = 2a\gamma_2 - 2c\gamma_1 -  b\beta_1 + d\beta_2;\\
\vspace{2mm} x(b) - y(d) = 2b\gamma_2 - 2d\gamma_1 -   a\beta_1 + c\beta_2;\\
\vspace{2mm} x(\beta_2) - y(\beta_1) + \gamma_1\beta_1 - \gamma_2\beta_2 = 2(bc - ad).
\end{array}
\end{equation}

If we suppose that $a=b$ and $c=d$ (or $a=-b$ and $c=-d$), we get a contradiction with the assumption $ad - bc \neq 0$.
If we suppose that $a=b$ and $c=-d$ (similarly for $a=-b$ and $c=d$), then from the first four equalities of \eqref{E:Int-Cond-1} we obtain 
$$\beta_1 = x (\ln |c|) - 2 \gamma_2; \quad \beta_2 = y (\ln |c|) - 2 \gamma_1.$$
On the other hand, $\gamma_1 = -y(\ln f), \gamma_2 = -x(\ln f)$. A direct calculation shows that $x(\beta_2) - y(\beta_1) + \gamma_1\beta_1 - \gamma_2\beta_2 = 0$. 
So, the last equality of \eqref{E:Int-Cond-1}  implies that $ac = 0$, which contradicts the 
assumption  that the first normal space is two-dimensional. 

Now, suppose that $a^2 - b^2 \neq 0$ and $c =  d$ (or $c = -d$). Then the first four equalities of \eqref{E:Int-Cond-1} imply
$$\beta_1 = x (\ln |a-b|) - 2 \gamma_2; \quad \beta_2 = y (\ln |a-b|) - 2 \gamma_1; \quad \text{in case} \; c =  d,$$
or 
$$\beta_1 = -x (\ln |a+b|) + 2 \gamma_2; \quad \beta_2 = -y (\ln |a+b|) + 2 \gamma_1;\quad \text{in  case} \; c =  -d.$$
The last expressions of $\beta_1$ and $\beta_2$ together with $\gamma_1 = -y(\ln f), \gamma_2 = -x(\ln f)$ imply that $x(\beta_2) - y(\beta_1) + \gamma_1\beta_1 - \gamma_2\beta_2 = 0$. Then from the last equality of \eqref{E:Int-Cond-1} we obtain $c(a-b) = 0$ (or $c(a+b)=0$), which contradicts the assumption  that the first normal space is two-dimensional. Similarly, in the case $a^2 = b^2$ and $c^2 - d^2 \neq 0$ we also get a contradiction.

Hence, $a^2 - b^2 \neq 0$ and $c^2 - d^2 \neq 0$ for each $(u,v) \in \mathcal{D}$, i.e.  $\sigma(x, x)$ and $\sigma(x, y)$ are both non-lightlike vector fields.
If $ac - bd = 0$ for each $(u,v) \in \mathcal{D}$, then $\langle \sigma(x,x), \sigma(x,y) \rangle = 0$ and hence  $\sigma(x, x)$ and $\sigma(x, y)$ are orthogonal. If $ac - bd \neq 0$ at a point $p =z(u_0, v_0)$ of $M^2_1$, then there exists a subdomain  $\mathcal{D}_0 \subset \mathcal{D}$, $(u_0, v_0) \in \mathcal{D}_0$,  such that  $ac - bd \neq 0$ in $\mathcal{D}_0 $. In this case, we consider an orthonormal tangent frame field $\bar{x}, \bar{y}$, defined in $\mathcal{D}_0$, such that 
\begin {equation*}
\begin {array} {l}
\vspace{2mm}\bar{x} = \cosh\varphi\,x + \sinh\varphi\,y;\\
\vspace{2mm}\bar{y} = \sinh\varphi\,x + \cosh\varphi\,y
\end {array}
\end {equation*}
for some smooth function  $\varphi$. Then 
\begin {equation*}
\begin {array} {l}
\vspace{2mm}\sigma(\bar{x}, \bar{x}) = \bar{a}\,e_1 + \bar{b}\,e_2;\\
\vspace{2mm}\sigma(\bar{x}, \bar{y}) = \bar{c}\,e_1 + \bar{d}\,e_2,
\end {array}
\end {equation*}
where the functions $\bar{a}, \bar{b}, \bar{c}, \bar{d}$ are expressed as follows
\begin {equation*}
\begin {array} {ll}
\vspace{2mm}\bar{a} = a\,\cosh2\varphi + c\,\sinh2\varphi; & \qquad\quad \bar{c} = a\,\sinh2\varphi + c\,\cosh2\varphi;\\
\vspace{2mm}\bar{b} = b\,\cosh2\varphi + d\,\sinh2\varphi; & \qquad\quad \bar{d} = b\,\sinh2\varphi + d\,\cosh2\varphi.
\end{array}
\end {equation*}
Straightforward computations show that 
$$(\bar{a}\bar{c} - \bar{b}\bar{d}) = \ds{\frac{a^2 + c^2 - b^2 - d^2}{2}}\,\sinh4\varphi +  (ac - bd)\,\cosh4\varphi.$$
If we suppose that $b^2 - a^2 + d^2 -c^2 = 0$, then $a^2 +c^2 = b^2 + d^2$ and hence there exist functions $p$, $q$, $r$, such that  
$a = r \cos p$, $c = r \sin p$, $b = r \cos q$, $d = r \sin q$. Then, using \eqref{E:Eq-2} we get 
$$K^2  - \varkappa^2 = - 4(ac -bd)^2,$$
which contradicts the condition $K^2  - \varkappa^2 > 0$.
Hence, we can assume that $b^2 - a^2 + d^2 -c^2 \neq 0$. Then, $\langle \sigma(\bar{x}, \bar{x}), \sigma(\bar{x}, \bar{y}) \rangle = 0$ if and only if there exists a function $\varphi$ such that 
 $$\tanh 4 \varphi = \ds{\frac{2(ac - bd)}{b^2 - a^2 + d^2 -c^2}}.$$
Denote $A = \ds{\frac{2(ac - bd)}{b^2 - a^2 + d^2 -c^2}}$. 
Then 
$$A -1 = \ds{\frac{(a+c)^2 - (b+d)^2}{b^2 - a^2 + d^2 -c^2}}; \qquad A +1 = \ds{\frac{-(a-c)^2 + (b-d)^2}{b^2 - a^2 + d^2 -c^2}},$$
 and after some computations we obtain 
$$A^2-1 = - \ds{\frac{(a^2 - b^2)^2 + (c^2 - d^2)^2 - 2(bc-ad)^2 - 2(ac-bd)^2}{(b^2 - a^2 + d^2 -c^2)^2}}.$$
Having in mind \eqref{E:Eq-2}, we get
$$A^2-1 = \ds{\frac{ - (K^2 - \varkappa^2)}{(b^2 - a^2 + d^2 -c^2)^2}}.$$
Since we study surfaces satisfying $K^2 - \varkappa^2  > 0$, we get that $-1 <A< 1$, $A \neq 0$ for  $(u, v) \in {\mathcal D}_0$.
So, there exists  a function $\varphi(u,v)$ defined in ${\mathcal D}_0$ such that $A = \tanh 4 \varphi$. Hence, $\langle \sigma(\bar{x}, \bar{x}), \sigma(\bar{x}, \bar{y}) \rangle = 0$, i.e. $\sigma(\bar{x}, \bar{x})$ and $\sigma(\bar{x}, \bar{y})$ are orthogonal. 
Consequently, there exists an orthonormal  normal frame field $\{n_1, n_2\}$ such that 
\begin {equation*}
\sigma(\bar{x}, \bar{x}) = \nu\,n_1;\qquad
\sigma(\bar{x}, \bar{y}) = \mu\,n_2
\end {equation*}
for some functions $\nu \neq 0$ and $\mu \neq 0$.

\vskip 2mm
Finally, we obtain that in a neighbourhood of any point of a  minimal Lorentz surface of general type 
we can introduce  a special  orthonormal frame field $\{x, y, n_1, n_2\}$
such that
\begin{equation*}\label{Eq-3}
\begin{array}{l}
\vspace{2mm}
\sigma(x,x) = \nu \,n_1; \\
\vspace{2mm}
\sigma(x,y) = \qquad \mu\,n_2;  \\
\vspace{2mm}
 \sigma(y,y) = \nu \,n_1,
\end{array}
\end{equation*}
where $\nu\mu \neq 0$ and $\langle x, x \rangle = 1$, $\langle y, y  \rangle= -1$, $\langle n_1, n_1 \rangle = \varepsilon$, $\langle n_2, n_2 \rangle = -\varepsilon$ ($\varepsilon = \pm 1$). We call this orthonormal frame field
 a \emph{geometric frame field} of the surface and the tangent directions determined by the tangent vector fields $x$ and $y$ 
we call \textit{canonical directions} of the surface. Obviously, the canonical directions are uniquely determined at any point of a
minimal Lorentz surface of general type. The normal frame field $\{n_1; n_2\}$ 
is also uniquely determined by the canonical tangents, and the functions $\mu$ and $\nu$ are geometric
functions  of the surface.

With respect to the geometric  frame field  $\{x, y, n_1, n_2\}$ the following Frenet-type derivative formulas hold true
\begin {equation} \label{E:Frenet-general}
\begin {array} {ll}
\vspace{2mm}\nabla'_{x}x=\quad\quad\;-\gamma_{1}\,y+\nu\,n_1; & \qquad\quad\nabla'_{x}n_1=-\varepsilon\nu\;x\qquad\quad\quad\quad+\beta_{1}\,n_2;\\
\vspace{2mm}\nabla'_{x}y=-\gamma_{1}\,x\quad\quad\quad\quad \;\;+\mu\,n_2; & \qquad\quad\nabla'_{y}n_1=\quad\quad\quad\varepsilon\nu\;y \qquad\quad+\beta_{2}\,n_2;\\
\vspace{2mm}\nabla'_{y}x=\quad\quad\;-\gamma_{2}\,y\quad\quad \;+\mu\,n_2; & \qquad\quad\nabla'_{x}n_2=\quad\quad-\varepsilon\mu\;y + \beta_{1}\,n_1;\\
\vspace{2mm}\nabla'_{y}y=-\gamma_{2}\,x\quad\quad\;+\nu\,n_1; & \qquad\quad\nabla'_{y}n_2=\varepsilon\mu\,x \quad\quad\quad+ \beta_{2}\,n_1.\end{array}
\end {equation}
Using that the Levi-Civita connection $\nabla'$ is flat, from \eqref{E:Frenet-general} we obtain that the functions 
$\gamma_1$,  $\gamma_2$, $\nu$, $\mu$, $\beta_1$, and $\beta_2$ satisfy  the following conditions 
\begin{equation} \label{E: Integr.cond.}
\begin{array} {l}
\vspace{2mm} 
x(\gamma_2) - y(\gamma_1) + \gamma_1^2 - \gamma_2^2 = -\varepsilon\,(\mu^2 + \nu^2);\\
\vspace{2mm} 
x(\beta_2) - y(\beta_1) + \gamma_1\,\beta_1 - \gamma_2\,\beta_2 = -2\varepsilon\,\mu\,\nu;\\
\vspace{2mm} 
x(\mu) = 2\mu\,\gamma_2 + \nu\,\beta_2;\\
\vspace{2mm} 
y(\mu) = 2\mu\,\gamma_1 + \nu\,\beta_1;\\
\vspace{2mm} 
x(\nu) = 2\nu\,\gamma_2 + \mu\,\beta_2;\\
\vspace{2mm} 
y(\nu) = 2\nu\,\gamma_1 + \mu\,\beta_1.
\end{array}
\end{equation}

The Gauss curvature $K$ and the normal curvature $\varkappa$ are expressed  in terms of the invariants $\mu$ and $\nu$ as follows
\begin{equation} \label{E:Kandkappa}
\begin{array} {l}
\vspace{2mm} K = -\varepsilon\,(\mu^2 + \nu^2);\\
\vspace{2mm} \varkappa = -2\,\mu\,\nu.
\end{array}
\end{equation}

Since $\nu\mu \neq 0$, we can formulate the following statement.

\begin{prop}\label{P:non-zero K}
Let $M^2_1$ be a Lorentz minimal surface of general type. Then at each point of the surface the Gauss curvature $K$ and 
the normal curvature $\varkappa$ are non-zero.
\end{prop}

It follows from \eqref{E:Kandkappa} that 
$$K^2 - \varkappa^2 = (\mu ^2 - \nu^2)^2.$$

\vskip 3mm
\noindent
\textbf{\textit{Remark}.} 
Minimal Lorentz surfaces satisfying $K^2 - \varkappa^2 = 0$ (or equivalently $\mu ^2 = \nu^2$)  are the analogue of  minimal super-conformal surfaces in the Euclidean space $\E^4$. 
Instead of ellipse of normal curvature  defined for surfaces in $\E^4$, in the pseudo-Euclidean space $\E^4_2$ we can use the notion of curvature hyperbola associated to the second fundamental form  of a Lorentz surface  (see \cite{BaPaBringas}). Using the terminology from the Euclidean space, we call a minimal Lorentz surfaces \textit{super-conformal} if at each point the Gauss curvature and the normal curvature satisfy the equality $K^2 - \varkappa^2 = 0$.

\vskip 3mm
In this paper we study surfaces for which $K^2 - \varkappa^2 >0$, so we assume that $\mu ^2 - \nu^2 \neq 0$.
Then, from \eqref{E: Integr.cond.} we obtain that
\begin {equation*}
\gamma_1 = \frac{1}{4} y \left(\ln\left|\mu^2 - \nu^2\right| \right), \qquad \gamma_2 = \frac{1}{4} x \left(\ln\left|\mu^2 - \nu^2\right| \right).
\end {equation*}
Taking into account that $\gamma_{1} = -\ds{\frac{f_v}{f^2}}$, $\gamma_{2} = -\ds{\frac{f_u}{f^2}}$, we get
\begin {equation*}
x \left(\left|\mu^2 - \nu^2\right| f^4  \right) = 0, \qquad y \left(\left|\mu^2 - \nu^2\right| f^4  \right) = 0,
\end {equation*}
which imply that the function  $\left|\mu^2 - \nu^2\right| f^4 $ is constant. Hence,  
 $$f^2 = \frac{c^2}{\sqrt{\left|\mu^2 - \nu^2\right|}},$$ 
where $c$ is a non-zero constant. After the change of the parameters
\begin {equation*}
\bar{u} = c\,u, \quad\quad \bar{v} = c\,v,
\end {equation*}
we may assume that 
\begin {equation} \label{E:canonical param}
f(u, v) = \ds{\frac{1}{\sqrt[4]{\left|\mu^2 - \nu^2\right|}}}.
\end {equation}

Following the terminology in \cite{GM}, we call the parameters $(u,v)$ \textit{canonical} if 
\begin {equation*} 
E(u, v) = - G(u,v) = \ds{\frac{1}{\sqrt{\left|\mu^2 - \nu^2\right|}}}; \qquad F(u,v) = 0.
\end {equation*}
Note that any minimal Lorentz surface  of general type  locally admits canonical parameters.

Now, we assume that the surface is parametrized by canonical parameters. Then by use of \eqref{E: Integr.cond.} and \eqref{E:canonical param} we obtain that 
the functions $\gamma_1$, $\gamma_2$, $\beta_1$,  and $\beta_2$ are expressed in terms of $\mu$ and $\nu$ as follows
\begin {equation} \label{E:gammabetta} 
\begin {array}{ll}
\vspace{6mm} \gamma_1 = \left(\sqrt[4]{\left|\mu^2 - \nu^2\right|}\right)_v, & \quad \gamma_2 = \left(\sqrt[4]{\left|\mu^2 - \nu^2\right|}\right)_u,\\
\vspace{2mm} 
\beta_1 = \ds{\frac{\sqrt[4]{\left|\mu^2 - \nu^2\right|}}{2}\,\left(\ln\left|\frac{\mu + \nu}{\mu - \nu}\right|\right)_v}, & 
\quad \beta_2 = \ds{\frac{\sqrt[4]{\left|\mu^2 - \nu^2\right|}}{2}\,\left(\ln\left|\frac{\mu + \nu}{\mu - \nu}\right|\right)_u}.
\end{array}
\end {equation}

\noindent
Using \eqref{E:gammabetta} and the first two equalities of \eqref{E: Integr.cond.} we obtain the following partial differential equations for the functions $\mu$ and $\nu$ 
\begin{equation} \label{E:Delta}
\begin{array}{l}
\vspace{2mm} 
\ds{\sqrt{\left|\mu^2 - \nu^2\right|}}\,\Delta^h \ln \left|\mu^2 - \nu^2\right| = - 4\varepsilon (\mu^2 + \nu^2)\\
\vspace{2mm} 
\ds{\sqrt{\left|\mu^2 - \nu^2\right|}}\,\Delta^h \ln\left|\frac{\mu + \nu}{\mu - \nu}\right| = -4\,\varepsilon\mu\,\nu,
\end{array}
\end{equation}
where $\Delta^h = \frac{\partial^2}{\partial u^2} - \frac{\partial^2}{\partial v^2}$ is the hyperbolic Laplace operator.

\vskip 3mm

Now we can prove the following  fundamental Bonnet-type theorem for the class of minimal Lorentz surfaces of general type.

\begin{thm} \label{T:Fundamental}
Let $\mu(u, v)$ and $\nu(u, v)$ be two smooth functions, defined in a domain ${\mathcal D}, \,\, {\mathcal D} \subset {\R}^2$, and satisfying the conditions:
\begin{equation*}
\begin{array}{l}
\vspace{2mm} 
\mu\,\nu \neq 0; \qquad\quad \mu^2 - \nu^2\neq0;\\
\vspace{2mm} 
\ds{\sqrt{\left|\mu^2 - \nu^2\right|}}\,\Delta^h \ln \left|\mu^2 - \nu^2\right| = - 4\varepsilon (\mu^2 + \nu^2);\\
\vspace{2mm} 
\ds{\sqrt{\left|\mu^2 - \nu^2\right|}}\,\Delta^h \ln\left|\frac{\mu + \nu}{\mu - \nu}\right| = -4\,\varepsilon\mu\,\nu,
\end{array}
\end{equation*}
where $\varepsilon = \pm 1$. Let $\{x_{0},\,y_{0},\,(n_1)_{0},\,(n_2)_{0}\}$ be an orthonormal frame at
a point $p_{0}\in\E^4_2$  such that $\langle x_0, x_0 \rangle = 1$, $\langle y_0, y_0 \rangle = -1$, $\langle (n_1)_{0}, (n_1)_{0} \rangle = \varepsilon$, $\langle (n_2)_{0}, (n_2)_{0} \rangle = -\varepsilon$. Then there exist a subdomain ${\mathcal D}_0 \subset {\mathcal D}$
and a unique minimal Lorentz surface of general type $M^2_1: z = z(u,v), \,\, (u,v) \in {\mathcal D}_0$, such that $M^2_1$ passes through $p_0$,  the functions  $\mu(u,v)$, $\nu(u,v)$ are the geometric functions of $M^2_1$, $\{x_{0},\,y_{0},\,(n_1)_{0},\,(n_2)_{0}\}$ is the geometric frame of $M^2_1$ at the point $p_0$, $\varepsilon = 1$ (resp. $\varepsilon = -1$) in the case the geometric normal vector field $n_1$ is spacelike (resp. timelike). Furthermore, $(u, v)$ are canonical parameters of the surface.
\end{thm}

\begin{proof} 
Let us define the following functions
\begin{equation} \label{E:Definitions}
\begin {array}{ll}
\vspace{2mm} 
E = \ds{\frac{1}{\sqrt{\left|\mu^2 - \nu^2\right|}}}; & \quad G =  \ds{-\frac{1}{\sqrt{\left|\mu^2 - \nu^2\right|}}};\\
\vspace{2mm} 
\gamma_1 = \left(\sqrt[4]{\left|\mu^2 - \nu^2\right|}\right)_v; & \quad \gamma_2 = \left(\sqrt[4]{\left|\mu^2 - \nu^2\right|}\right)_u;\\
\vspace{2mm} 
\beta_1 = \ds{\frac{\sqrt[4]{\left|\mu^2 - \nu^2\right|}}{2}\,\left(\ln\left|\frac{\mu + \nu}{\mu - \nu}\right|\right)_v}; & 
\quad \beta_2 = \ds{\frac{\sqrt[4]{\left|\mu^2 - \nu^2\right|}}{2}\,\left(\ln\left|\frac{\mu + \nu}{\mu - \nu}\right|\right)_u}.
\end{array}
\end{equation}
We consider the following system of partial
differential equations for the unknown vector functions $x=x(u,v),\,y=y(u,v),\,n_1=n_1(u,v),\,n_2=n_2(u,v)$
in $\E_{2}^{4}$
\begin{equation}\label{E:eq3}
\begin{array}{ll}
\vspace{2mm}
x_{u} = - \sqrt{E}\gamma_{1}\,y + \sqrt{E} \nu\,n_1 & \qquad x_{v} = - \sqrt{-G}\gamma_{2}\,y+ \sqrt{-G}\mu\,n_2\\
\vspace{2mm}
y_{u} = - \sqrt{E}\gamma_{1}\,x + \sqrt{E} \mu\,n_2 & \qquad y_{v} = - \sqrt{-G}\gamma_{2}\,x+ \sqrt{-G} \nu\,n_1\\
\vspace{2mm}
(n_1)_{u} = - \sqrt{E}\varepsilon\,\nu\,x + \sqrt{E} \beta_{1}\,n_2 & \qquad (n_1)_{v} =  \sqrt{-G}\varepsilon\,\nu\,y + \sqrt{-G}\beta_{2}\,n_2\\
\vspace{2mm}
(n_2)_{u} = - \sqrt{E} \varepsilon\,\mu\,y + \sqrt{E} \beta_{1}\,n_1 & \qquad (n_2)_{v} = \sqrt{-G}\varepsilon\,\mu\,x + \sqrt{-G}\beta_{2}\,n_1
\end{array}
\end{equation}
Denote
$$Z=\left(\begin{array}{c}
\vspace{2mm} 
x\\
\vspace{2mm}
y\\
\vspace{2mm} 
n_1\\
\vspace{2mm} 
n_2
\end{array}\right);\quad
A=\sqrt{E}\left(\begin{array}{cccc}
\vspace{2mm} 0 & -\gamma_{1} & \nu & 0\\
\vspace{2mm} -\gamma_{1} & 0 & 0 & \mu\\
\vspace{2mm} -\varepsilon\,\nu & 0 & 0 & \beta_{1}\\
\vspace{2mm} 0 & -\varepsilon\,\mu & \beta_{1} & 0
\end{array}\right);$$

$$B=\sqrt{-G}\left(\begin{array}{cccc}
\vspace{2mm} 0 & -\gamma_{2} & 0 & \mu\\
\vspace{2mm} -\gamma_{2} & 0 & \nu & 0\\
\vspace{2mm} 0 & \varepsilon\,\nu & 0 & \beta_{2}\\
\vspace{2mm} \varepsilon\,\mu & 0 & \beta_{2} & 0
\end{array}\right).$$
Using  matrices $A$ and $B$ we can rewrite  system \eqref{E:eq3} in the form
\begin {equation}\label{E:eq4}
\begin{array}{l}
\vspace{2mm}Z_{u}=A\,Z;\\
\vspace{2mm}Z_{v}=B\,Z.
\end{array}
\end{equation}
The integrability conditions of system \eqref{E:eq4} are
$$Z_{uv}=Z_{vu},$$
or equivalently,
\begin {equation}\label{E:eq5}
{\displaystyle {\frac{\partial a_{i}^{k}}{\partial v}-\frac{\partial b_{i}^{k}}{\partial u}+\sum_{j=1}^{4}(a_{i}^{j}\,b_{j}^{k}-b_{i}^{j}\,a_{j}^{k})=0,\quad i,k=1,\dots,4,}}
\end {equation}
where $a_{i}^{j}$ and $b_{i}^{j}$ are the elements of the matrices
$A$ and $B$. By use of \eqref{E:Definitions} and  equations \eqref{E:Delta} it can be checked that  equalities \eqref{E:eq5} are fulfilled.
Hence, there exist a subset $\mathcal{D}_{1}\subset\mathcal{D}$
and unique vector functions $x=x(u,v),\,y=y(u,v),\,n_1=n_1(u,v),\,n_2=n_2(u,v),\,\,(u,v)\in\mathcal{D}_{1}$,
which satisfy system \eqref{E:eq3} and the initial conditions
$$x(u_{0},v_{0})=x_{0},\quad y(u_{0},v_{0})=y_{0},\quad n_1(u_{0},v_{0})=(n_1)_{0},\quad n_2(u_{0},v_{0})=(n_2)_{0}.$$

Now, we shall prove that the vectors $x(u,v),\,y(u,v),\,n_1(u,v),\,n_2(u,v)$
form an orthonormal frame in $\E_{2}^{4}$ for each $(u,v)\in\mathcal{D}_{1}$.
Let us consider the following functions
$$\begin{array}{lll}
\vspace{2mm}\varphi_{1}=\langle x,x\rangle-1; & \qquad\varphi_{5}=\langle x,y\rangle; & \qquad\varphi_{8}=\langle y,n_1\rangle;\\
\vspace{2mm}\varphi_{2}=\langle y,y\rangle+1; & \qquad\varphi_{6}=\langle x,n_1\rangle; & \qquad\varphi_{9}=\langle y,n_2\rangle;\\
\vspace{2mm}\varphi_{3}=\langle n_1,n_1\rangle-\varepsilon; & \qquad\varphi_{7}=\langle x,n_2\rangle; & \qquad\varphi_{10}=\langle n_1,n_2\rangle;\\
\vspace{2mm}\varphi_{4}=\langle n_2,n_2\rangle+\varepsilon;
\end{array}$$
defined for  $(u,v)\in\mathcal{D}_{1}$.
Since $x(u,v),\,y(u,v),\,n_1(u,v),\,n_2(u,v)$
satisfy \eqref{E:eq3}, for the functions $\varphi_i(u,v), \,\,i = 1, \dots, 10$ we obtain the following system
$$\begin{array}{lll}
\vspace{2mm}{\displaystyle {\frac{\partial\varphi_{i}}{\partial u}=\alpha_{i}^{j}\,\varphi_{j}},}\\
\vspace{2mm}{\displaystyle {\frac{\partial\varphi_{i}}{\partial v}=\beta_{i}^{j}\,\varphi_{j}};}
\end{array}\qquad i=1,\dots,10,$$
where $\alpha_{i}^{j},\beta_{i}^{j},\,\,i,j=1,\dots,10$ are functions of $(u,v)\in\mathcal{D}_{1}$.
This is a linear system of partial differential equations
 satisfying the initial conditions $\varphi_i(u_0,v_0) = 0, \,\,i = 1,
\dots, 10$.
 Hence, $\varphi_{i}(u,v)=0,\,\,i=1,\dots,10$ for each $(u,v)\in\mathcal{D}_{1}$.
Consequently, the vector functions $x(u,v),\,y(u,v),\,n_1(u,v),\,n_2(u,v)$ form an orthonormal frame in $\E_{2}^{4}$ for each $(u,v)\in\mathcal{D}_{1}$.

Finally, we consider the following system of partial differential equations for the vector function $z=z(u, v)$
\begin{equation}\label{E:eq6}
\begin{array}{l}
\vspace{2mm}z_{u}=\sqrt{E}\,x\\
\vspace{2mm}z_{v}=\sqrt{-G}\,y
\end{array}
\end{equation}
Using \eqref{E:eq3} we obtain that the integrability conditions $z_{uv}=z_{vu}$ of system \eqref{E:eq6} are fulfilled.
Hence, there exist a subset $\mathcal{D}_{0}\subset\mathcal{D}_{1}$
and a unique vector function $z=z(u,v)$, defined for $(u,v)\in\mathcal{D}_{0}$
and satisfying $z(u_{0},v_{0})=p_{0}$.

Consequently, the surface $M^2_1:z=z(u,v),\,\,(u,v)\in\mathcal{D}_{0}$
satisfies the assertion of the theorem.

\end{proof}

The meaning of Theorem \ref{T:Fundamental} is that
\textit{any minimal Lorentz surface of general type  is determined up to a rigid motion in $\E^4_2$ by two  invariant functions $\mu$ and $\nu$ 
satisfying the system of  natural partial differential equations  \eqref{E:Delta}}.

\vskip 3mm

Using equalities \eqref{E:Kandkappa} we can express the  functions $\mu$ and $\nu$ in terms of the Gauss curvature $K$ and the normal curvature $\varkappa$. More precisely,  the following relations hold true
\begin{equation*}
\begin{array}{l}
\vspace{2mm} 
\ds{\left|\mu^2 - \nu^2\right|} = \ds{\sqrt{K^2 - \varkappa^2}};\\
\vspace{2mm} 
\ds{\left|\frac{\mu + \nu}{\mu - \nu}\right|} = \ds{\sqrt{\frac{K + \varepsilon\,\varkappa}{K - \varepsilon\,\varkappa}}}.
\end{array}
\end{equation*}
Hence, equations  \eqref{E:Delta} can be rewritten in terms of  $K$ and $\varkappa$ as follows
\begin{equation} \label{E:Delta_new}
\begin{array}{l}
\vspace{2mm} 
\ds{\sqrt[4]{K^2 - \varkappa^2}}\,\Delta^h \ln(K^2 - \varkappa^2)  = 8 K\\
\vspace{2mm} 
\ds{\sqrt[4]{K^2 - \varkappa^2}}\,\Delta^h \ln \frac{K + \varepsilon\,\varkappa}{K - \varepsilon\,\varkappa}  = 4 \varepsilon\,\varkappa.
\end{array}
\end{equation}

Then the fundamental theorem for minimal Lorentz surfaces of general type can be stated in terms of the curvatures $K$ and $\varkappa$ as follows.

\begin{thm} \label{T:Fundamental-2}
Let $K(u, v) \neq 0$ and $\varkappa(u, v) \neq 0$ be two smooth functions, 
defined in a domain ${\mathcal D}, \,\, {\mathcal D} \subset {\R}^2$, such that  $K^2 - \varkappa^2 > 0$, and satisfying the equations
\begin{equation*} 
\begin{array}{l}
\vspace{2mm} 
\ds{\sqrt[4]{K^2 - \varkappa^2}}\,\Delta^h  \ln (K^2 - \varkappa^2)  = 8 K\\
\vspace{2mm} 
\ds{\sqrt[4]{K^2 - \varkappa^2}}\,\Delta^h \ln \frac{K + \varepsilon\,\varkappa}{K - \varepsilon\,\varkappa}  = 4 \varepsilon\,\varkappa
\end{array}
\end{equation*}
where $\varepsilon = \pm 1$.
Then there exists a unique (up to a rigid motion) minimal Lorentz surface of general type such that $K(u, v)$  and $\varkappa(u, v)$ are the Gauss curvature and the normal curvature, respectively; $\varepsilon = 1$ (resp. $\varepsilon = - 1$) in the case the geometric normal vector field $n_1$ is spacelike (resp. timelike). Furthermore,  $(u,v)$ are the canonical parameters of the surface.
\end{thm}

Finally, the background system of partial differential equations for minimal Lorentz surfaces of general type in $\E^4_2$ is system
\eqref{E:Delta_new} or equivalently, system \eqref{E:Delta}.

\section{Example}

In this section we shall construct an example of a minimal Lorentz surface of general type applying  the procedure given in the proof of Theorem \ref{T:Fundamental} to a concrete solution to system \eqref{E:Delta}.

Let us consider the following functions
$$\mu(u) = \frac{2}{(1-4 \cos^2 \frac{u}{\sqrt[4]{3}})^{\frac{3}{2}}}; \qquad \nu(u) = \frac{1+ 2 \cos^2 \frac{u}{\sqrt[4]{3}}}{\sin^2 \frac{u}{\sqrt[4]{3}} (1 - 4 \cos^2 \frac{u}{\sqrt[4]{3}})^{\frac{3}{2}}},$$
where $u \in \left(\sqrt[4]{3}\, \frac{\pi}{3}; 2 \sqrt[4]{3} \,\frac{\pi}{3}\right)$. For simplicity we denote $t = \frac{u}{\sqrt[4]{3}}$, 
$t \in \left(\frac{\pi}{3}; 2\frac{\pi}{3}\right)$.
Using the functions $\mu$ and $\nu$ defined above, we obtain
$$\mu^2 - \nu^2 = \frac{3}{\sin^4 t (1 - 4 \cos^2 t)^2}; \qquad \frac{\mu+\nu}{\mu-\nu} = \frac{3}{1- 4 \cos^2t}.$$
Direct computations show that $\mu(u)$ and $\nu(u)$ satisfy system  \eqref{E:Delta} in the case $\varepsilon = -1$.
Let us consider the following functions
\begin{equation*} \label{E:Def-Ex}
\begin {array}{ll}
\vspace{2mm} 
E = \ds{\frac{\sin^2 t(1-4\cos^2 t)}{\sqrt{3}}}; & \quad G =  - \ds{\frac{\sin^2 t(1-4\cos^2 t)}{\sqrt{3}}};\\
\vspace{2mm} 
\gamma_1 = 0; & \quad \gamma_2 = - \ds{\frac{\cos t(5-8\cos^2 t)}{\sin^2 t(1-4\cos^2 t)^{\frac{3}{2}}}};\\
\vspace{2mm} 
\beta_1 = 0; & 
\quad \beta_2 = - \ds{\frac{4 \cos t}{(1-4\cos^2 t)^{\frac{3}{2}}}}.
\end{array}
\end{equation*}
Now, system \eqref{E:eq3} for the vector functions $x=x(u,v),\,y=y(u,v),\,n_1=n_1(u,v),\,n_2=n_2(u,v)$ takes the form
\begin{equation*}
\begin{array}{ll}
\vspace{2mm}
x_{u} = \frac{1+2 \cos^2t}{\sqrt[4]{3} \sin t (1-4\cos^2 t)}\,n_1 & \quad x_{v} = \frac{\cos t(5-8\cos^2 t)}{\sqrt[4]{3} \sin t (1-4\cos^2 t)}\,y + \frac{2 \sin t}{\sqrt[4]{3} (1-4\cos^2 t)}\,n_2\\
\vspace{2mm}
y_{u} = \frac{2 \sin t}{\sqrt[4]{3} (1-4\cos^2 t)}\,n_2 & \quad y_{v} = \frac{\cos t(5-8\cos^2 t)}{\sqrt[4]{3} \sin t (1-4\cos^2 t)}\,x+ \frac{1+2 \cos^2t}{\sqrt[4]{3} \sin t (1-4\cos^2 t)}\,n_1\\
\vspace{2mm}
(n_1)_{u} = \frac{1+2 \cos^2t}{\sqrt[4]{3} \sin t (1-4\cos^2 t)}\,x  & \quad (n_1)_{v} = - \frac{1+2 \cos^2t}{\sqrt[4]{3} \sin t (1-4\cos^2 t)}\,y - \frac{4 \sin t\cos t}{\sqrt[4]{3} (1-4\cos^2 t)}\,n_2\\
\vspace{2mm}
(n_2)_{u} = \frac{2 \sin t}{\sqrt[4]{3} (1-4\cos^2 t)} \,y  & \quad (n_2)_{v} = - \frac{2 \sin t}{\sqrt[4]{3} (1-4\cos^2 t)}\,x - \frac{4 \sin t\cos t}{\sqrt[4]{3} (1-4\cos^2 t)}\,n_1
\end{array}
\end{equation*}

Denoting 
\begin{equation} \label{E:eq2-ex}
\begin{array}{ll}
\vspace{2mm}
\varphi(t) = \frac{1+2 \cos^2t}{\sqrt[4]{3} \sin t (1-4\cos^2 t)}; & \qquad \psi(t) = \frac{2 \sin t}{\sqrt[4]{3} (1-4\cos^2 t)};\\
\vspace{2mm}
p(t) = \frac{\cos t(5-8\cos^2 t)}{\sqrt[4]{3} \sin t (1-4\cos^2 t)}; &\qquad q(t) = \frac{4 \sin t\cos t}{\sqrt[4]{3} (1-4\cos^2 t)},
\end{array}
\end{equation}
we rewrite the above system as follows
\begin{equation}\label{E:eq3-ex}
\begin{array}{ll}
\vspace{2mm}
x_{u} = \varphi(t) \,n_1 & \qquad \quad x_{v} = p(t)\,y + \psi(t)\,n_2\\
\vspace{2mm}
y_{u} = \psi(t)\,n_2 & \qquad \quad y_{v} = p(t)\,x+ \varphi(t)\,n_1\\
\vspace{2mm}
(n_1)_{u} = \varphi(t) \,x  & \qquad \quad (n_1)_{v} = - \varphi(t)\,y - q(t)\,n_2\\
\vspace{2mm}
(n_2)_{u} = \psi(t)\,y  & \qquad \quad (n_2)_{v} = - \psi(t) \,x - q(t)\,n_1
\end{array}
\end{equation}
Let us consider the vector functions $X = x-n_1$ and $Y = y -n_2$. It follows from \eqref{E:eq3-ex} that 
$$ X_u = -\varphi X; \qquad Y_u = -\psi Y.$$
Hence, for each coordinate function $X^k$ ($k = 1, \dots, 4$) of $X$ we have $X^k_u = - \varphi X^k$. The last equality implies that 
\begin{equation}\label{E:eq4-ex}
\ln X^k = -\int{\varphi(u) du} + C^k(v),
\end{equation} 
where $C^k(v)$, $k = 1, \dots, 4$ is a function of $v$. Calculating the integral in \eqref{E:eq4-ex} we obtain
$$X^k = \frac{\cos 2t - \cos t}{\sin t \sqrt{1 - 4 \cos^2 t}}\, \xi^k(v), \quad k = 1, \dots, 4$$
for some functions $\xi^k(v)$. We denote $a(t)= \frac{\cos 2t - \cos t}{\sin t \sqrt{1 - 4 \cos^2 t}}$. Then the vector function $X(u,v)$ is expressed as 
\begin{equation}\label{E:eq4-X}
X(u,v) = a(\frac{u}{\sqrt[4]{3}}) \, \xi(v)
\end{equation}
for some vector function $\xi(v)$.

Similarly, for each coordinate function $Y^k$ ($k = 1, \dots, 4$) of $Y$ we have $Y^k_u = - \psi Y^k$, which implies that 
\begin{equation}\label{E:eq5-ex}
\ln Y^k = -\int{\psi(u) du} + D^k(v),
\end{equation} 
where $D^k(v)$, $k = 1, \dots, 4$ is a function of $v$. Calculating the integral in \eqref{E:eq5-ex} we get
$$Y^k = \frac{1+2\cos t}{\sqrt{1 - 4 \cos^2 t}}\, \eta^k(v), \quad k = 1, \dots, 4,$$
 $\eta^k(v)$ being a function of $v$. Denoting $b(t)= \frac{1+2\cos t}{\sqrt{1 - 4 \cos^2 t}}$, we obtain that the vector function $Y(u,v)$ is expressed as 
\begin{equation}\label{E:eq5-Y}
Y(u,v) = b(\frac{u}{\sqrt[4]{3}}) \, \eta(v)
\end{equation}
for some vector function $\eta(v)$.

Having in mind that $x$, $y$, $n_1$, $n_2$ should satisfy $\langle x, x\rangle = 1$, $\langle y, y\rangle = 1$, $\langle n_1, n_1\rangle = -1$, $\langle n_2, n_2\rangle = 1$, we get $\langle X, X\rangle = 0$, $\langle Y, Y\rangle = 0$, $\langle X, Y\rangle = 0$, and hence 
$\langle \xi, \xi\rangle = 0$, $\langle \eta, \eta \rangle = 0$, $\langle \xi, \eta\rangle = 0$. Let us consider the following vector functions
\begin{equation}\label{E:eq6-ex}
\begin{array}{l}
\vspace{2mm}
\xi(v) = \left(\cos \frac{2v}{\sqrt[4]{3}}, \sin \frac{2v}{\sqrt[4]{3}}, - \cos \frac{v}{\sqrt[4]{3}}, -\sin \frac{v}{\sqrt[4]{3}} \right);\\
\vspace{2mm}
\eta(v) = \left(-\sin \frac{2v}{\sqrt[4]{3}}, \cos \frac{2v}{\sqrt[4]{3}}, - \sin \frac{v}{\sqrt[4]{3}}, \cos \frac{v}{\sqrt[4]{3}} \right).
\end{array}
\end{equation} 

Since $X = x - n_1$, $Y = y - n_2$,  from \eqref{E:eq4-X} and \eqref{E:eq5-Y} we get 
\begin{equation}\label{E:eq7-ex}
x = a(t) \,\xi(v) + n_1; \qquad y = b(t) \, \eta(v) + n_2,
\end{equation}
which imply $x_v = a \, \xi_v + (n_1)_v$, $y_v = b \, \eta_v + (n_2)_v$. Now, using \eqref{E:eq3-ex} and \eqref{E:eq7-ex}, we obtain
\begin{equation}\label{E:eq8-ex}
\begin{array}{l}
\vspace{2mm}
(\varphi + \psi +p + q) \,n_1 = b \,\eta_v - (p+ \psi)\, a \, \xi;\\
\vspace{2mm}
(\varphi + \psi +p + q)\,n_2 = a \,\xi_v - (p+ \varphi)\, b \, \eta.
\end{array}
\end{equation}
Equalities \eqref{E:eq8-ex} together with  \eqref{E:eq2-ex} and \eqref{E:eq6-ex} allow us to find the vector functions $n_1$ and $n_2$. They are expressed as follows
\begin{equation}\label{E:eq9-ex}
\begin{array}{l}
\vspace{2mm}
n_1(t,s) = \frac{1}{\sin t \sqrt{1-4\cos^2 t}} \left(\cos t \cos 2s, \cos t \sin 2s, \cos 2t \cos s, \cos 2t \sin s\right);\\
\vspace{2mm}
n_2(t,s) = \frac{1}{\sin t \sqrt{1-4\cos^2 t}} \left(\sin t \sin 2s, -\sin t \cos 2s, \sin 2t \sin s, -\sin 2t \cos s\right),
\end{array}
\end{equation}
where $ s = \frac{v}{\sqrt[4]{3}}$.

Using  \eqref{E:eq7-ex} and \eqref{E:eq9-ex} we find the vector functions $x$ and $y$:
\begin{equation}\label{E:eq10-ex}
\begin{array}{l}
\vspace{2mm}
x(t,s) = \frac{1}{\sin t \sqrt{1-4\cos^2 t}} \left(\cos 2t \cos 2s, \cos 2t \sin 2s, \cos t \cos s, \cos t \sin s\right);\\
\vspace{2mm}
y(t,s) = \frac{1}{\sin t \sqrt{1-4\cos^2 t}} \left(-\sin 2t \sin 2s, \sin 2t \cos 2s, -\sin t \sin s, \sin t \cos s\right).
\end{array}
\end{equation}

Now we consider system  \eqref{E:eq6}. In our example it takes the form
\begin{equation*}
\begin{array}{l}
\vspace{2mm}
z_u = \frac{\sin t \sqrt{1-4\cos^2 t}} {\sqrt[4]{3}} \, x\\
\vspace{2mm}
z_v = \frac{\sin t \sqrt{1-4\cos^2 t}} {\sqrt[4]{3}} \, y\\
\end{array}
\end{equation*}
Using that  $t = \frac{u}{\sqrt[4]{3}}$, $s = \frac{v}{\sqrt[4]{3}}$, we rewrite the above system as
\begin{equation}\label{E:eq11-ex}
\begin{array}{l}
\vspace{2mm}
z_t = \sin t \sqrt{1-4\cos^2 t} \, x\\
\vspace{2mm}
z_s = \sin t \sqrt{1-4\cos^2 t} \, y\\
\end{array}
\end{equation}
Now,  system \eqref{E:eq11-ex} together with \eqref{E:eq10-ex} imply 
\begin{equation}\label{E:eq12-ex}
z(t,s) =  \left( \frac{1}{2}\sin 2t \cos 2s, \frac{1}{2}\sin 2t \sin 2s, \sin t \cos s, \sin t \sin s\right) + C,
\end{equation}
where $C$ is a constant vector. 

The vector function $z(t,s)$ given by \eqref{E:eq12-ex} determines a minimal Lorentz surface of general type in $\E^4_2$. Note that $(t,s)$ are not the canonical parameters of the surface.  With respect to canonical parameters the surface is  parametrized as follows
\begin{equation*}\label{E:eq13-ex}
z(u,v) =  \left( \frac{1}{2}\sin 2\frac{u}{\sqrt[4]{3}} \cos 2\frac{v}{\sqrt[4]{3}}, \frac{1}{2}\sin 2\frac{u}{\sqrt[4]{3}} \sin 2\frac{v}{\sqrt[4]{3}}, \sin \frac{u}{\sqrt[4]{3}} \cos \frac{v}{\sqrt[4]{3}}, \sin \frac{u}{\sqrt[4]{3}} \sin \frac{v}{\sqrt[4]{3}}\right) + C.
\end{equation*}

Finally, given a concrete solution $(\mu, \nu)$ to system \eqref{E:Delta}  we obtained an example of a minimal Lorentz surface of general type parametrized by canonical parameters.

\vskip 5mm \textbf{Acknowledgments:}
The  authors are partially supported by the National Science Fund,
Ministry of Education and Science of Bulgaria under contract DN 12/2.

\end{document}